\documentclass[]{interact}

\usepackage{amsmath,amssymb,enumerate,subfigure,multirow,hhline,color}
\usepackage{xcolor}
\usepackage{hyperref}
\usepackage{graphicx}
\usepackage{graphics}
\usepackage[sort]{cite}
\usepackage{amsthm}
\usepackage{algorithm}
\usepackage{algpseudocode}

\usepackage{tcolorbox}
\usepackage{epsfig,psfrag}
\usepackage{tabularx}
\usepackage{tabulary}

\usepackage[sort]{cite}

\usepackage{hyperref}
\hypersetup{breaklinks=true,colorlinks=true,linkcolor=blue,citecolor=blue}

\usepackage[noabbrev,capitalise,nameinlink]{cleveref}

\DeclareMathOperator{\argmin}{arg\,min}

\newtheorem{theorem}{Theorem}
\newtheorem{example}{Example}
\newtheorem{assumption}{Assumption}

\newtheorem{problem}{Problem}
\newtheorem{lemma}{Lemma}
\newtheorem{definition}{Definition}
\newtheorem{proposition}{Proposition}

\allowdisplaybreaks

\begin{document}


\title{On the Semidefinite Duality of Finite-Horizon LQG Problem}

\author{
\name{Donghwan Lee \thanks{Email: donghwan@kaist.ac.kr}}
\affil{Department of Electrical Engineering,
KAIST, Daejeon, 34141, South Korea}
}

\maketitle

\begin{abstract}
In this paper, our goal is to study fundamental foundations of linear quadratic Gaussian (LQG) control problems for stochastic linear time-invariant systems via Lagrangian duality of semidefinite programming (SDP) problems. In particular, we derive an SDP formulation of the finite-horizon LQG problem, and its Lagrangian duality. Moreover, we prove that Riccati equation for LQG can be derived the KKT optimality condition of the corresponding SDP problem. Besides, the proposed primal problem efficiently decouples the system matrices and the gain matrix. This allows us to develop new convex relaxations of non-convex structured control design problems such as the decentralized control problem. We expect that this work would provide new insights on the LQG problem and may potentially facilitate developments of new formulations of various optimal control problems. Numerical examples are given to demonstrate the effectiveness of the proposed methods.
\end{abstract}


\section{Introduction}

Duality has long been a core concept in optimal control theory such as the Pontryagin's maximal principle. On the other hand, emergence of convex optimization~\cite{Boyd2004} and semidefinite programming (SDP) techniques in control analysis and design promoted new optimization formulations of control problems~\cite{Boyd1994,de1999new,el2000advances,geromel2007h,de2002extended,li2002robust,xu2008survey,scherer2000linear,palhares2000mixed,han2005new,lee2010less} during the last decades. Accordingly, the corresponding dual problems have been studied to further deepen our understanding of the classical control theories, e.g.,~\cite{Boyd1994,de2002extended,yao2001stochastic,rami2000linear}. For instance, a new proof of Lyapunov's matrix inequality was presented in~\cite{henrion2001rank} based on
the standard SDP duality~\cite{vandenberghe1996semidefinite}. In addition, SDP formulations
of the LQR problem and their dual formulations were developed in~\cite{yao2001stochastic} and~\cite{rami2000linear}. Comprehensive studies on the SDP dualities in systems and control theory, such
as the Kalman-Yakubovich-Popov (KYP) lemma, the LQR problem, and
the $H_\infty$-norm computation, were provided in~\cite{balakrishnan2003semidefinite}. A new Lagrangian duality result and its relation to reinforcement learning problems were established in~\cite{lee2018primal} for infinite-horizon LQR problems. More recent results include the state-feedback solution to the LQR problem~\cite{gattami2010generalized}, the generalized KYP lemma and $H_\infty$ analysis~\cite{you2013lagrangian,you2015primal} derived using the Lagrangian duality, a sufficient condition for the strong duality of non-convex SDP problems~\cite{lee2021lossless}.

Several relations between the LQG problems and SDP problems have been studied in the literature, see for example,~\cite{Boyd1994,de2002extended,yao2001stochastic,rami2000linear}. The recent paper,~\cite{gattami2010generalized}, proposed a new SDP formulation, where the finite-horizon LQR problem was converted into the optimal covariance matrix selection problem, and it can be also interpreted as a dual problem of the standard LQR approaches based on the Riccati equations or the Lyapunov methods.

In this paper, we consider the finite-horizon linear quadratic Gaussian (LQG) control problem~\cite{bertsekas1996neuro}. The goal is to investigate a new semidefinite programming (SDP) formulation of the finite-horizon LQG problem and its dual counterpart by using the Lagrangian duality in standard convex optimization~\cite{Boyd2004}. In particular, we prove that Riccati equation for LQG can be derived the KKT optimality condition of the corresponding SDP problem. Moreover, the proposed primal problem efficiently decouples the system matrices and the gain matrix. This fact allows us to develop new convex relaxations of non-convex structured control design problems such as the decentralized control problem.
We expect that the results in this paper provide new insights on the LQG problem based on a relation between our primal and dual formulations and the Riccati equation, which can potentially facilitate developments new algorithms and new formulations of various optimal control problems, such as the data-drive control design algorithm~\cite{lee2018primal}.

{\bf Notation}: The adopted notation is as follows: ${\mathbb N}$
and ${\mathbb N}_+$: sets of nonnegative and positive integers,
respectively; ${\mathbb R}$: set of real numbers; ${\mathbb R}_+$:
set of nonnegative real numbers; ${\mathbb R}_{++}$: set of
positive real numbers; ${\mathbb R}^n $: $n$-dimensional Euclidean
space; ${\mathbb R}^{n \times m}$: set of all $n \times m$ real
matrices; $A^T$: transpose of matrix $A$; $A \succ 0$ ($A \prec
0$, $A\succeq 0$, and $A\preceq 0$, respectively): symmetric
positive definite (negative definite, positive semi-definite, and
negative semi-definite, respectively) matrix $A$; $I_n $: $n
\times n$ identity matrix; ${\mathbb S} ^n $: symmetric $n \times
n$ matrices; ${\mathbb S}_+^n $: cone of symmetric $n \times n$
positive semi-definite matrices; ${\mathbb S}_{++}^n$:
symmetric $n\times n$ positive definite matrices; $Tr(A)$:
trace of matrix $A$; ${\cal N}(v,W)$: normal distribution with mean $v$ and variance $W$; $*$ inside a matrix: transpose of its symmetric term; s.t.: subject to.

\section{Preliminaries}
In this section, we briefly summarize basic concepts of the standard Lagrangian duality theory in~\cite{Boyd2004}. Let us consider the optimization problem with a matrix inequality (semidefinite programming, SDP), which is our main concern in this paper.
\begin{problem}[Primal problem]\label{prob:1}
Solve for $x \in {\mathbb R}^n$
\begin{align*}
p^*:=&\inf_{x \in \cal D} f(x)\quad {\rm{s.t.}}\quad \Phi(x) \preceq 0
\end{align*}
where $x\in {\mathbb R}^n$, $\Phi: {\mathbb R}^n \to {\mathbb S}^{\hat n}$ is a differentiable matrix function, $f: {\mathbb R}^n \to {\mathbb R}$ is a differentiable objective function, and  ${\cal D} \subseteq {\mathbb R}^n$ is some convex set.
\end{problem}
Note that in~\cref{prob:1}, we use $\inf$ instead of $\max$ because $\cal D$ can be potentially an open set.
An important property of problems of the form in~\cref{prob:1} that arises frequently is the convexity.
\begin{definition}[Convexity]
\cref{prob:1} is said to be convex if $f$ is a convex function, $\cal D$ is a convex set, and the feasible set, $\{ x \in {\mathbb R}^n :\Phi(x) \preceq 0\}$, is convex.
\end{definition}
Note that for the feasible set, $\{ x \in {\mathbb R}^n :\Phi(x) \preceq 0\}$, to be convex, $\Phi(x)$ needs to be linear or convex in $x$. Associated with~\cref{prob:1}, the Lagrangian function~\cite{Boyd2004} is defined as
\[
L(x,\Lambda) := f (x) + Tr(\Lambda \Phi(x))
\]
for any $\Lambda \in {\mathbb S}^{\hat n}_+$, called the Lagrangian multiplier. For any $\Lambda \in {\mathbb S}^{\hat n}_+$, we define the dual function as
\[
g(\Lambda) := \inf_{x \in \cal D} L(x,\Lambda ) = \inf _{x \in \cal D} (f (x) + Tr(\Lambda \Phi(x)))
\]
It is known that the dual function yields lower bounds on the optimal value $p^*$:
\begin{align}
g(\Lambda ) \le p^*\label{eq:13}
\end{align}
for any Lagrange multiplier, $\Lambda \in {\mathbb S}^{\hat n}_+$. The Lagrange dual problem associated with~\cref{prob:1} is defined as follows.
\begin{problem}[Dual problem]\label{prob:2}
Solve for $\Lambda \in {\mathbb S}^{\hat n}_+$
\begin{align*}
&d^*:=\sup_{\Lambda \in {\mathbb S}^{\hat n}_+} g (\Lambda).
\end{align*}
\end{problem}
The dual problem is known to be convex even if the primal is not. In this context, the original~\cref{prob:1} is sometimes called the primal problem. Similarly, $d^*$ is called the dual optimal value, while $p^*$ is called the primal optimal value. The inequality~\eqref{eq:13} implies the important inequality
\[
d^*  \le p^*,
\]
which holds even if the original problem is not convex. This property is called weak duality, and the difference, $p^* - d^*$ is called the optimal duality gap. If the equality $d^* = p^*$ holds, i.e., the optimal duality gap is zero, then we say that strong duality holds.
\begin{definition}[Strong duality]
If the equality, $d^* = p^*$, holds, then we say that strong duality holds for~\cref{prob:1}.
\end{definition}
There are many results that establish conditions on the problem under
which strong duality holds. These conditions are called constraint qualifications. Once such constraint qualification is Slater's condition, which is stated below.
\begin{lemma}[Slater's condition]
Suppose that~\cref{prob:1} is convex. If there exists an $x\in  {\bf relint}({\cal D})$ such that
\[
\Phi (x) \prec 0,
\]
then the strong duality holds, where ${\bf relint}({\cal D})$ is the relative interior~\cite[pp.~37]{Boyd2004}.
\end{lemma}

Without the constrain qualifications, such as the Slater's condition, the strong duality does not hold.
For more comprehensive discussions on the duality, the reader is referred to the monograph~\cite{Boyd2004}. Before closing this section, we introduce several transformations of matrix inequalities, which will play important roles in this paper, and are thus summarized in this section. One of the most popular transformations of matrix inequalities is the so-called Schur complement, which frequently arises in LMI-based computational control designs.
\begin{lemma}[Schur complement~{\cite{Boyd1994}}]\label{lemma:Schur-complement}
The matrix inequality
\[
A^T H A - P \prec 0,\quad H \succ 0
\]
holds if and only if
\[
\left[ {\begin{array}{*{20}c}
   { - P} & * \\
   {A} & {-H^{-1}}  \\
\end{array}} \right] \prec 0
\]
\end{lemma}

In~\cite{de1999new}, an extended Schur complement was introduced to deal with robust control design problems. For convenience, it is outlined below.
\begin{lemma}[Extended Schur complement~I,~{\cite{de1999new}}]\label{lemma:extended-Schur-complement-I}
The following conditions are equivalent:
\begin{enumerate}
\item For a symmetric matrix $P\succ 0$, the matrix inequality
\[
A^T P A - P \prec 0
\]
holds.

\item For a symmetric matrix $P\succ 0$, there exist a matrix $G$ such that
\[
\left[ {\begin{array}{*{20}c}
   { - P} & *  \\
   {GA} & {P - G - G^T }  \\
\end{array}} \right] \prec 0
\]
\end{enumerate}
\end{lemma}

Application of~\cref{lemma:extended-Schur-complement-I} is restricted in the sense that $P$ appears twice on the left-hand side of the inequalities. In this paper, we will use a generalized version of~\cref{lemma:extended-Schur-complement-I}, which eliminates this restriction.
Since the proof is not presented in the literature, it is briefly presented here for completeness.
\begin{lemma}[Extended Schur complement~II]\label{lemma:extended-Schur-complement-II}
The following conditions are equivalent:
\begin{enumerate}
\item For symmetric matrices $H\succ 0$ and $P\succ 0$, the matrix inequality
\[
A^T H A - P \prec 0
\]
holds.

\item There exists a matrix $G$ such that
\[
\left[ {\begin{array}{*{20}c}
   { - P} & *\\
   {GA} & {H - G - G^T }  \\
\end{array}} \right] \prec 0
\]
\end{enumerate}
\end{lemma}
\begin{proof}
Suppose 1) holds. Then, using the Schur complement~\cref{lemma:Schur-complement}, we have
\[
\left[ {\begin{array}{*{20}c}
   { - P} & * \\
   {A} & {-H^{-1}}  \\
\end{array}} \right] \prec 0
\]
Multiplying the last inequality with $\left[ {\begin{array}{*{20}c}
   I & 0  \\
   0 & H  \\
\end{array}} \right]$ from the left and right, one gets
\[
\left[ {\begin{array}{*{20}c}
   { - P} & * \\
   {HA} & {-H}  \\
\end{array}} \right] = \left[ {\begin{array}{*{20}c}
   { - P} & * \\
   {HA} & {H-H-H}  \\
\end{array}} \right] \prec 0
\]
Letting $G = H$, we conclude that 2) holds. Conversely, assume that 2) holds. Then,
\begin{align*}
&\left[ {\begin{array}{*{20}c}
   I  \\
   A  \\
\end{array}} \right]^T \left[ {\begin{array}{*{20}c}
   { - P} & * \\
   {GA} & {H - G - G^T }  \\
\end{array}} \right]\left[ {\begin{array}{*{20}c}
   I  \\
   A  \\
\end{array}} \right]\\
=&\left[ {\begin{array}{*{20}c}
   I  \\
   A  \\
\end{array}} \right]^T \left\{ {\left[ {\begin{array}{*{20}c}
   { - P} & 0  \\
   0 & H  \\
\end{array}} \right] + }  {\left[ {\begin{array}{*{20}c}
   0  \\
   G  \\
\end{array}} \right]\left[ {\begin{array}{*{20}c}
   A & { - I}  \\
\end{array}} \right] + \left[ {\begin{array}{*{20}c}
   A & { - I}  \\
\end{array}} \right]^T \left[ {\begin{array}{*{20}c}
   0  \\
   G  \\
\end{array}} \right]^T } \right\}\left[ {\begin{array}{*{20}c}
   I  \\
   A  \\
\end{array}} \right]\\
=&  - P + A^T HA\\
\prec & 0.
\end{align*}

Therefore, 1) holds. This completes the proof.
\end{proof}

\cref{lemma:extended-Schur-complement-II} is more general than~\cref{lemma:extended-Schur-complement-I} in the sense that if $H = P$,~\cref{lemma:extended-Schur-complement-I} is recovered from~\cref{lemma:extended-Schur-complement-II}.
Lastly, the matrix inequality transformations presented here have a common feature: they present two matrix inequalities which are equivalent in some sense. The concept of equivalent relations will be more rigorously formalized in the next section.

\section{Equivalence and strong duality}\label{sec:strong-duality}
In this section, we present the concept of the equivalent transformation and its relation with the strong duality proposed in~\cite{lee2021lossless}. Consider the following transformed optimization of the original problem,~\cref{prob:1}.
\begin{problem}[Transformation~I]\label{prob:3}
Solve for $(x,z) \in {\mathbb R}^n \times {\mathbb R}^m$
\begin{align*}
\hat p^*:=& \inf_{(x,z) \in {\cal D} \times {\mathbb R}^m} f(x)\quad {\rm{s.t.}}\quad \Theta(x,z) \preceq 0
\end{align*}
where $x \in {\mathbb R}^n$, $\Theta: {\mathbb R}^n \times {\mathbb R}^m \to {\mathbb S}^{\hat n + \hat m}$ is a differentiable matrix function, and is a transformation of $\Phi(x)\in {\mathbb S}^{\hat n} $, and $z \in {\mathbb R}^m$ is an additional variable introduced through the transformation.
\end{problem}

\cref{prob:3} and~\cref{prob:1} can be related via some property called the equivalence, which is defined below.
\begin{definition}[Equivalence]\label{def:equivalence}
For any $t\in {\mathbb R}$ and $U\in {\mathbb S}^{\hat n}$, define the two sets
\begin{align*}
{\cal F}(U,t): =& \{ x \in {\cal D},\Phi (x) \prec U,f(x) \le t\}\\
{\cal G}(U,t): =& \left\{ {(x,z) \in {\cal D} \times {\mathbb R}^n ,\Theta (x,z) \prec \left[ {\begin{array}{*{20}c}
   U & 0  \\
   0 & 0  \\
\end{array}} \right],f(x) \le t} \right\}.
\end{align*}

The two problems,~\cref{prob:1} and~\cref{prob:3}, are said to be equivalent if the following two statements are true:
\begin{enumerate}
\item for any $t\in {\mathbb R}$ and $U\in {\mathbb S}^{\hat n}$ such that ${\cal F}(U,t) \neq\emptyset$ and $x \in {\cal F}(U,t)$, there exist $z \in {\mathbb R}^n$ such that $(x,z) \in {\cal G}(U,t)$

\item for any $t\in {\mathbb R}$ and $U\in {\mathbb S}^{\hat n}$ such that ${\cal G}(U,t) \neq\emptyset$ and $(x,z) \in {\cal F}(G,t)$, $x\in {\cal G}(U,t)$ holds.
\end{enumerate}
\end{definition}

An implication of the strong equivalence in~\cref{def:equivalence} is that an optimal solution of one problem can be recovered from an optimal solution of the other problem and vice versa. This concept is formalized below.
\begin{lemma}
Suppose that the Slater's condition holds for~\cref{prob:1}. Moreover, suppose that~\cref{prob:1} and~\cref{prob:3} are strongly equivalent. Then, $p^* = \hat p^*$ holds. Moreover, let $x^*$ be an optimal solution of~\cref{prob:1}. Then, there exists some $\hat z^*$ such that $(x^*,\hat z^*)$ is an optimal solution of~\cref{prob:3}. Conversely, if $(\hat x^*,\hat z^*)$ is an optimal solution of~\cref{prob:3}, then $\hat x^*$ is an optimal solution of~\cref{prob:1}.
\end{lemma}

In the following, we study a convexification of matrix inequality constrained optimizations, which have a special property to be addressed soon. Toward this goal, let us consider the following optimization problem.
\begin{problem}[Transformation~II]\label{prob:4}
Solve
\begin{align*}
&\inf_{(v,w) \in h({\cal D}\times {\mathbb R}^m) } g(v,w)\quad {\rm{s.t.}}\quad \Omega (v,w) \preceq 0
\end{align*}
for some mapping $h$ such that $h({\cal D}\times {\mathbb R}^m)$ is convex, where $\Omega: {\mathbb R}^n \times {\mathbb R}^m \to {\mathbb S}^{\hat n+\hat m}$ and $g: {\mathbb R}^n \times {\mathbb R}^m \to {\mathbb R}$ are convex, and $f$ and $\Theta$ can be expressed as
\begin{align*}
f (x) =& g (h(x,z)) = (g  \circ h)(x,z)\\
\Theta(x,z) =& \Omega(h(x,z)) = (\Omega \circ h)(x,z)
\end{align*}
\end{problem}

Note that~\cref{prob:4} is convex, and hence will be called a convexification of~\cref{prob:3}.
In particular, we will consider a special convexification called lossless convexification defined below.
\begin{definition}[Lossless convexification]\label{def:lossless-convexification}
Define the following sets associated with~\cref{prob:3} and~\cref{prob:4}:
\begin{align}
{\cal F}: =& \{ (x,z)\in {\cal D} \times {\mathbb R}^m:,\Theta(x,z) \prec 0\},\label{eq:24}\\
{\cal G}:=& \left\{ {(v,w) \in h({\cal D}\times {\mathbb R}^n): \Omega(v,w) \prec 0} \right\},\label{eq:25}
\end{align}
and suppose that $h$ is such that $h:{\cal F} \to {\cal G}$ is bijection.
Then,~\cref{prob:3} is said to be a lossless convexification of~\cref{prob:2}.
\end{definition}

An implication of~\cref{def:lossless-convexification} is that solutions of~\cref{prob:4} have bijective correspondences to solutions of~\cref{prob:3}. Therefore, even if~\cref{prob:3} is nonconvex, its solutions can be found from the convex~\cref{prob:4}.
Moreover, another property is that the existence of such a lossless convexification ensures the strong duality of the original~\cref{prob:1} (with the Slater's condition).
This result is formally summarized below.
\begin{lemma}[Strong duality]\label{thm:strong-duality2-matrix}
We suppose that~\cref{prob:3} and~\cref{prob:1} are equivalent, and~\cref{prob:4} is a lossless convexification of~\cref{prob:3}. If~\cref{prob:1} satisfies the Slater's condition, then the strong duality holds for~\cref{prob:1}.
\end{lemma}

\section{Finite-horizon LQG problem}
In this section, we turn our attention to the optimal control problem, which is our main concern in this paper.
Consider the stochastic linear time-invariant (LTI) system
\begin{align}
&x(k+1)=Ax(k)+Bu(k)+w(k),\label{stochastic-LTI-system2}
\end{align}
where $k\in {\mathbb N}$, $x(k)\in {\mathbb R}^n$ is the state
vector, $u(k)\in {\mathbb R}^m$ is the input vector, $x(0)\sim
{\cal N}(0,W_f)$ and $w(k)\sim {\cal N}(0,W)$ with $W_f\succ 0$ and $W\succ 0$ are mutually independent Gaussian random vectors. In this paper, we consider the following finite-horizon linear quadratic Gaussian (LQG) problem:
\begin{problem}[Finite-horizon LQG problem]\label{finite-horizon-LQR-problem}
Solve
\begin{align*}
&\min_{F_0,\ldots,F_{N-1}\in
{\mathbb R}^{m \times n}}\,{\mathbb E}(x(k)^T Q_f x(k))+ \sum_{k=0}^{N-1}{{\mathbb E}\left(\begin{bmatrix}
   x(k)\\ u(k)\\
\end{bmatrix}^T \begin{bmatrix}
   Q & 0 \\ 0 & R  \\
\end{bmatrix} \begin{bmatrix}
   x(k)\\ u(k) \\
\end{bmatrix} \right)}\\
&{\rm s.t.}\quad x(k+1)=Ax(k)+Bu(k)+w(k),\quad u(k)=F_k x(k).
\end{align*}
\end{problem}

A collection of assumptions that will be used throughout the paper
is summarized below.
\begin{assumption}\label{assumption:1}
In this paper, we assume that $Q_f \succeq 0,Q \succeq 0,R \succ 0,W_f \succ 0$, and
$W \succ 0$.
\end{assumption}

If we define the covariance of the augmented vector $[x(k)^T,u(k)^T]^T\in {\mathbb R}^{n \times m}$
\begin{align*}
&S_k={\mathbb E}\left( \begin{bmatrix} x(k)\\ u(k)\\ \end{bmatrix} \begin{bmatrix}
   x(k)\\
   u(k)\\
\end{bmatrix}^T \right),\quad k \in \{0,\ldots,N\},
\end{align*}
then,~\cref{finite-horizon-LQR-problem} can be equivalently
converted to the matrix equality constrained optimization problem.
\begin{problem}\label{primal-LQR}
Solve
\begin{align*}
&J_p^*:=\min_{S_0,\ldots,S_{N-1}\in {\mathbb S}^{n+m},F_0,\ldots,F_{N-1}\in {\mathbb R}^{m \times n}}\,J_p(\{S_k\}_{k=0}^{N-1})\\
&{\rm s.t.}\quad \Phi (F_k,S_{k-1})=S_k\quad k \in \{1,\ldots,N-1\},\quad \begin{bmatrix}
   I_n\\
   F_0\\
\end{bmatrix} W_f \begin{bmatrix}
   I_n\\
   F_0\\
\end{bmatrix}^T = S_0,
\end{align*}
where
\begin{align*}
J_p(\{S_k\}_{k=0}^{N-1}):=& Tr\left(Q_f
\left(\begin{bmatrix}
   A^T\\B^T\\
\end{bmatrix}^T S_{N - 1} \begin{bmatrix}
   A^T\\B^T\\
\end{bmatrix}+W\right) \right)+\sum_{k=0}^{N-1}{ Tr\left( \begin{bmatrix}
 Q & 0\\ 0 & R\\
\end{bmatrix} S_k\right)}\\
\Phi (F, S):=&\begin{bmatrix}
   I_n \\F\\
\end{bmatrix}\left( \begin{bmatrix}
   A^T\\ B^T\\
\end{bmatrix}^T S \begin{bmatrix}
   A^T\\B^T\\
\end{bmatrix}+W\right)\begin{bmatrix}
   I_n\\F\\
\end{bmatrix}^T
\end{align*}
\end{problem}

In~\cref{primal-LQR}, the matrix equality constraints represent the covariance updates. In this paper, instead of dealing with~\cref{primal-LQR} in its present form, we will consider the modified problem (SDP relaxation) by replacing the matrix equalities in~\cref{primal-LQR} by inequalities.
\begin{problem}[Primal problem]\label{primal-LQR-inequality-formulation}
Solve
\begin{align*}
&p_{\rm opt}:=\min_{S_0,\ldots,S_{N-1}\in {\mathbb S}^{n+m}, F_0,\ldots,F_{N-1}\in {\mathbb R}^{m \times
n}}\,\,J_p(\{S_k\}_{k=0}^{N-1})\\
&{\rm s.t.}\quad \Phi (F_k,S_{k-1})\preceq S_k,\quad k \in \{
1,\ldots,N-1\},\quad \begin{bmatrix}
   I\\F_0\\
\end{bmatrix} W_f \begin{bmatrix}
   I\\ F_0\\
\end{bmatrix}^T \preceq S_0.
\end{align*}
\end{problem}

Note that~\cref{primal-LQR-inequality-formulation} is not convex due to the bilinear matrix inequality constraints. We will study its solution through the Lagrangian duality. To this end, its Lagrangian dual problem can be derived as follows.
\begin{problem}[Dual problem~I]\label{dual-problem1}
Solve
\begin{align*}
d_{{\rm opt}}:=& \sup_{P_k\succeq 0,{\bar P}_k \succeq 0} g(\{(P_k
,{\bar P}_k )\}_{k=0}^{N-1})\\
=& \sup_{P_k \succeq 0,{\bar P}_k \succeq 0} \inf_{\{ S_k
,F_k\}_{k=0}^{N-1}} L(\{(S_k,F_k,P_k,{\bar P}_k)\}_{k=0}^{N-1}),
\end{align*}
where
\begin{align*}
&g(\{(P_k,{\bar P}_k)\}_{k=0}^{N-1}):=\inf_{\{S_k,F_k\}_{k=0}^{N-1}} L(\{(S_k,F_k,P_k,{\bar P}_k)\}_{k=0}^{N-1}),
\end{align*}
and $(P_k,{\bar P}_k)_{k=0}^{N-1}$ are called the dual variables.
\end{problem}

It is well-known that the dual problem is convex even if the primal is not~\cite{Boyd2004}.
In this paper, we will prove that the dual problem can be converted to an equivalent convex SDP problem.

\section{Main results}
To proceed, denote by ${\cal S}$ the set of all optimal solutions of the form $\{(F_k,S_k)\}_{k=0}^{N-1}$ of~\cref{primal-LQR-inequality-formulation}. In addition, define the mapping
${\cal F}: = \{ \{F_k\}_{k=0}^{N-1}:\{(F_k,S_k)\}_{k=0}^{N-1}\in {\cal S}\}$. We conclude that \cref{primal-LQR} is equivalent to~\cref{primal-LQR-inequality-formulation} in the following sense: if $\{F_k\}_{k=0}^{N-1}\in {\cal F}$, then it is also optimal for~\cref{primal-LQR}. This result is formally stated in the following proposition.
\begin{proposition}\label{optimal-property}
Let $\{F_k\}_{k=0}^{N-1}\in {\cal F}$. Then, it is an optimal solution of~\cref{primal-LQR}, and $J_p^*=p_{{\rm opt}}$ holds.
\end{proposition}
\begin{proof}
Let $\{ F_k \} _{k = 0}^{N - 1}  \in {\cal F}$ and construct $\{
{\bar S}_k \} _{k = 0}^{N - 1} $ such that
\begin{align*}
&\Phi (F_k ,{\bar S}_{k - 1} ) = {\bar S}_k ,\quad k
\in \{ 1,2, \ldots ,N - 1\},\quad \begin{bmatrix}
   I  \\
   {F_0 }  \\
\end{bmatrix} W_f \begin{bmatrix}
   I  \\
   {F_0 }  \\
\end{bmatrix}^T  = {\bar S}_0
\end{align*}

Clearly, ${\bar S}_k  \preceq S_k ,\forall k \in \{
0,1, \ldots ,N - 1\}$ and hence, $p_{{\rm opt}}  \ge J_p (\{
{\bar S}_k \} _{k = 0}^{N - 1} )$. However, since $\{ F_k
,{\bar S}_k \} _{k = 0}^{N - 1}$ is also a feasible point
of~\cref{primal-LQR-inequality-formulation}, and thus, $J_p (\{ {\bar S}_k \}
_{k = 0}^{N - 1} ) \ge p_{{\rm{opt}}}$. Therefore, $J_p (\{
{\bar S}_k \} _{k = 0}^{N - 1} ) = p_{\rm opt}$ and $\{
F_k ,{\bar S}_k\}_{k=0}^{N - 1} $ is an optimal
solution of~\cref{primal-LQR-inequality-formulation}. Since~\cref{primal-LQR} has a feasible set included by the
feasible set of~\cref{primal-LQR-inequality-formulation}, and the optimal solution
$\{F_k,{\bar S}_k \}_{k=0}^{N - 1} $ of~\cref{primal-LQR-inequality-formulation} takes equalities in the constraints of~\cref{primal-LQR-inequality-formulation}, $\{ F_k ,{\bar S}_k\} _{k 0}^{N-1} $ is also optimal solution of~\cref{primal-LQR}. The second statement is derived directly from the first statement. This
completes the proof.
\end{proof}

From~\cref{optimal-property}, we can conclude that~\cref{primal-LQR-inequality-formulation} can replace~\cref{finite-horizon-LQR-problem}. Therefore, in the sequel, we will address~\cref{primal-LQR-inequality-formulation} instead of~\cref{finite-horizon-LQR-problem}. For any $P_0,\ldots,P_{N-1}\in {\mathbb S}_+^{n+m}$, and ${\bar P}_0,\ldots,{\bar P}_{N-1}\in {\mathbb S}_+^{n+m}$, define the Lagrangian function of~\cref{primal-LQR-inequality-formulation}
\begin{align*}
L(\{(S_k,F_k,P_k,{\bar P}_k)\}_{k=0}^{N-1}):=& J_p (\{S_k\}_{k=0}^{N-1})+\sum_{k=1}^{N-1} {Tr((\Phi(F_k,S_{k-1})-S_k) P_k)}\\
&+Tr\left( \left(\begin{bmatrix}
   I\\F_0\\
\end{bmatrix}W_f \begin{bmatrix}
   I\\ F_0\\
\end{bmatrix}^T- S_0\right) P_0 \right)-\sum_{k=0}^{N-1}{Tr( S_k {\bar
P}_k)}
\end{align*}

Rearranging some terms, it can be represented by
\begin{align}
L(\{( S_k,F_k,P_k,{\bar P}_k)\}_{k=0}^{N-1})=& J_d (\{P_k,F_k\}_{k=0}^{N-1})\nonumber\\
&+ Tr\left(\left( \begin{bmatrix}
   A^T\\ B^T\\
\end{bmatrix} Q_f \begin{bmatrix}
   A^T\\B^T\\
\end{bmatrix}^T - P_{N-1} + \begin{bmatrix}
   Q & 0 \\ 0 & R\\
\end{bmatrix}-{\bar P}_{N-1}\right) S_{N-1}
\right)\nonumber\\
& + \sum_{k=1}^{N-1} {Tr((\Gamma(F_k,P_k)-P_{k-1}-{\bar P}_{k-1}) S_{k-1})}.\label{eq1}
\end{align}

The corresponding Lagrangian dual problem~\cite[chapter 5]{Boyd2004} is~\cref{dual-problem1}. In the following two theorems, we establish a relation between the dual optimal solution and the Riccati equation.
\begin{theorem}[Strong duality]\label{theorem1}
The strong duality holds, i.e., $p_{\rm opt}=d_{\rm opt} $;
\end{theorem}
\begin{proof}
To prove the strong duality, we will use the results in~\cref{sec:strong-duality}.
We will first prove that~\cref{primal-LQR-inequality-formulation} is strictly feasible to apply~\cref{thm:strong-duality2-matrix}.
With $F_k  = 0,\forall k \in \{
0,1, \ldots,N - 1\}$ and any $\varepsilon > 0$, construct
matrices $\{ S_k \} _{k = 0}^{N - 1}$ as follows:
\begin{align*}
&\begin{bmatrix}
  {I_n }  \\
  {F_0 }  \\
\end{bmatrix} W_f \begin{bmatrix}
 {I_n }  \\
 {F_0 }  \\
\end{bmatrix}^T  + \varepsilon I_n  = S_0,\quad \Phi (F_k ,S_{k - 1} ) + \varepsilon I_n = S_k
\end{align*}
The set $\{ F_k ,S_k \} _{k = 0}^{N - 1}$ satisfies the
constraints of~\cref{primal-LQR} with strict inequalities.
Therefore, we conclude that~\cref{primal-LQR-inequality-formulation} is strictly feasible.

Next, we will prove that the constraints in~\cref{primal-LQR-inequality-formulation}
can be equivalently converted to linear matrix inequality constraints. In particular, to apply the extended Schur complement,~\cref{lemma:extended-Schur-complement-II}, we first replace the non-strict matrix inequality ``$\preceq$'' and strict matrix inequality ``$\prec$'', and replace ``$\min$'' with ``$\inf$,'' which do not change the result. Then, by~\cref{lemma:extended-Schur-complement-II}, we have that $\Phi (F_k ,S_{k - 1} ) \prec S_k$ holds if and only if there exists $G_k  \in {\mathbb
R}^{n \times n}$ such that
\begin{align}
&\left[ {\begin{array}{*{20}c}
   {S_k } & *  \\
   {\left[ {\begin{array}{*{20}c}
   {G_k } & {G_k F_k^T }  \\
\end{array}} \right]} & {G_k  + G_k^T  - \left[ {\begin{array}{*{20}c}
   {A^T }  \\
   {B^T }  \\
\end{array}} \right]^T S_{k - 1} \left[ {\begin{array}{*{20}c}
   {A^T }  \\
   {B^T }  \\
\end{array}} \right] - W}  \\
\end{array}} \right] \succ 0\label{eq0}
\end{align}

Similarly, $\begin{bmatrix}
   I\\F_0\\
\end{bmatrix} W_f \begin{bmatrix}
   I\\ F_0\\
\end{bmatrix}^T \prec S_0$ is equivalent to
\begin{align}
\left[ {\begin{array}{*{20}c}
   {S_0 } & *  \\
   {\left[ {\begin{array}{*{20}c}
   {G_0 } & {G_0 F_0^T }  \\
\end{array}} \right]} & {G_0  + G_0^T  - W_f }  \\
\end{array}} \right]\succ 0.\label{eq2}
\end{align}
Next, the strict matrix inequality ``$\prec$'' can be replaced with the non-strict matrix inequality ``$\preceq$'', and ``$\min$'' can be replaced with ``$\inf$.'' Therefore, we see that~\cref{primal-LQR-inequality-formulation} is equivalent to
\begin{align*}
&p_{\rm opt}:=\min_{S_0,\ldots,S_{N-1}\in {\mathbb S}^{n+m}, F_0,\ldots,F_{N-1}\in {\mathbb R}^{m \times
n}}\,\,J_p(\{S_k\}_{k=0}^{N-1})\quad {\rm s.t.}\quad \eqref{eq0},\eqref{eq2}
\end{align*}
in the sense of~\cref{def:equivalence}. In the feasible set, $G_k$ is nonsingular. Therefore, we can find the bijective mapping
\[
h:\left[ {\begin{array}{*{20}c}
   S  \\
   {S'}  \\
   G  \\
   F  \\
\end{array}} \right] \mapsto \left[ {\begin{array}{*{20}c}
   S  \\
   {S'}  \\
   G  \\
   {FG^T }  \\
\end{array}} \right]
\]
to change variables
\[
h\left( {\left[ {\begin{array}{*{20}c}
   {S_k }  \\
   {S_{k - 1} }  \\
   {G_k }  \\
   {F_k }  \\
\end{array}} \right]} \right) = \left[ {\begin{array}{*{20}c}
   {S_k }  \\
   {S_{k - 1} }  \\
   {G_k }  \\
   {F_k G_k^T }  \\
\end{array}} \right] = \left[ {\begin{array}{*{20}c}
   {S_k }  \\
   {S_{k - 1} }  \\
   {G_k }  \\
   {H_k }  \\
\end{array}} \right],
\]
and~\eqref{eq0} and~\eqref{eq2} can be converted to the LMIs
\[
\left[ {\begin{array}{*{20}c}
   {S_k } & *  \\
   {\left[ {\begin{array}{*{20}c}
   {G_k } & {H_k }  \\
\end{array}} \right]} & {G_k  + G_k^T  - \left[ {\begin{array}{*{20}c}
   {A^T }  \\
   {B^T }  \\
\end{array}} \right]^T S_{k - 1} \left[ {\begin{array}{*{20}c}
   {A^T }  \\
   {B^T }  \\
\end{array}} \right] - W}  \\
\end{array}} \right]\succeq 0
\]
and
\[
\left[ {\begin{array}{*{20}c}
   {S_0 } & *  \\
   {\left[ {\begin{array}{*{20}c}
   {G_0 } & {H_0 }  \\
\end{array}} \right]} & {G_0  + G_0^T  - W_f }  \\
\end{array}} \right] \succeq 0
\]

Now, we can invoke~\cref{thm:strong-duality2-matrix} to prove that the strong duality holds for~\cref{primal-LQR-inequality-formulation}.
In particular, according to~\cref{thm:strong-duality2-matrix}, if~\cref{primal-LQR-inequality-formulation} is strictly feasible, and~\cref{primal-LQR-inequality-formulation} admits an equivalent convex SDP through a bijective mapping of variables, then it satisfies the strong duality. This completes the proof.
\end{proof}

According to~\cite[Chap.~5.5, pp.~243]{Boyd2004}, for any optimization problem with differentiable objective and constraint functions for which strong duality obtains, any pair of primal and dual optimal points must satisfy the KKT conditions. Since the strong duality holds for~\cref{primal-LQR-inequality-formulation}, we can obtain some information on the solution using the KKT condition. One result is that the Riccati equation can be derived from the KKT condition.
\begin{theorem}\label{theorem1b}
Consider the Riccati equation
\begin{align}
&A^T X_{k+1} A-A^T X_{k+1} B(R+B^T X_{k+1} B)^{-1} B^T X_{k+1} A+Q=X_k\label{Riccati}
\end{align}
for all $k \in \{0,\ldots,N-1\} $ with $X_N=Q_f$,
and define $\{( S_k,F_k,P_k,{\bar P}_k)\}_{k=0}^{N-1}$ with
\begin{align}
F_k=&-(R+B^T X_{k+1} B)^{-1} B^T X_{k+1} A,\nonumber\\
S_k=& \Phi (F_k,S_{k-1}),\quad S_0=\begin{bmatrix}
I\\ F_0\\
\end{bmatrix} W_f \begin{bmatrix}
 I\\ F_0\\
\end{bmatrix}^T,\nonumber\\
P_k =& \begin{bmatrix}
   Q + A^T X_{k+1} A & A^T X_{k+1} B\\
   B^T X_{k+1} A & R + B^T X_{k+1} B\\
\end{bmatrix},\nonumber\\
{\bar P}_k=& 0,\quad k \in \{0,\ldots,N-1\},
\label{primal-optimal-point}
\end{align}

Then, $\{(S_k,F_k)\}_{k=0}^{N-1}$ is an primal optimal point of~\cref{primal-LQR-inequality-formulation} and $\{(P_k,{\bar P}_k)\}_{k=0}^{N-1}$ is the corresponding dual optimal point of~\cref{primal-LQR-inequality-formulation}.
\end{theorem}
\begin{proof}
From the KTT condition of the generalized inequality constrained optimization in \cite[chap~5.9.2]{Boyd2004}, its
KKT condition can be summarized as the primal feasibility condition
\begin{align*}
&\begin{bmatrix}
   I  \\
   {F_0 }  \\
\end{bmatrix} W_f \begin{bmatrix}
   I  \\
   {F_0 }  \\
\end{bmatrix}^T  \preceq S_0 ,\quad \Phi (F_k ,S_{k - 1} ) \preceq
S_k,\quad k \in \{ 1,2, \ldots ,N - 1\}
\end{align*}
the complementary slackness condition
\begin{align}
\begin{matrix}& Tr\begin{pmatrix}\begin{pmatrix}\begin{bmatrix}
   I  \\
   {F_0 }  \\
\end{bmatrix}W_f \begin{bmatrix}
   I  \\
   {F_0 }  \\
\end{bmatrix}^T  - S_0\end{pmatrix} P_0 \end{pmatrix} =
0\\
&Tr((\Phi (F_k ,S_{k - 1} ) - S_k
) P_k ) = 0\\
&k \in \{ 1,2,\ldots ,N-1\}\\
&Tr( S_k {\bar P}_k ) = 0,\quad k \in \{ 0,1,
\ldots ,N- 1\}
\end{matrix}\label{complementary-slackness-condition}
\end{align}
and the dual feasibility condition
\begin{align*}
&P_N  = \begin{bmatrix}
   {Q_f } & 0  \\
   0 & 0  \\
\end{bmatrix},\quad \Gamma (0,P_N ) - {\bar P}_{N -
1} = P_{N-1}\\
&\Gamma (F_k ,P_k ) - {\bar P}_{k - 1}  =
P_{k - 1} ,\quad k \in \{ 1,2, \ldots N - 1\}\\
&W_f (P_{0,12}  + F_0^T P_{0,22}) + (P_{0,12}^T +
P_{0,\,22} F_0 )W_f  = 0\\
&M_k (P_{k + 1,12}  + F_{k + 1}^T P_{k + 1,22} ) + (P_{k + 1,12}^T  + P_{k + 1,22} F_{k + 1} )M_k  = 0\\
&k \in \{ 1,2, \ldots ,N - 1\}\\
&P_k \succeq 0,\quad {\bar P}_k  \succeq 0,k \in \{
0,1, \ldots N - 1\}
\end{align*}
where $M_k  = \begin{bmatrix}
   A & B  \\
\end{bmatrix} S_k \begin{bmatrix}
   A & B  \\
\end{bmatrix}^T + W$. By~\cref{assumption:1}, $M_k$ and $W_f$ are nonsingular, and hence, solving the KKT condition, we can prove that~\eqref{primal-optimal-point} uniquely solves the KKT condition. According to~\cite[Chap.~5.5, pp.~243]{Boyd2004}, for any optimization problem with differentiable objective and constraint functions for which strong duality obtains, any pair of primal and dual optimal points must satisfy the KKT conditions. Therefore, the point in~\eqref{primal-optimal-point} is the primal and dual optimal points of~\eqref{primal-LQR-inequality-formulation}. This completes the proof.
\end{proof}

\cref{theorem1} and~\cref{theorem1b} tell us that the optimal primal and dual solutions can be constructed using the solution of the Riccati equation. Conversely, the solution of the Riccati equation can be recovered from the optimal primal and dual solutions.

The dual problem in~\cref{dual-problem1} is a min-max problem, which is in general harder to solve than a minimization or maximization problem. Another dual formulation of~\cref{dual-problem1} is represented by a constrained maximization as follows:
\begin{problem}[Dual problem~II]\label{dual-LQR-inequality-formulation}
Solve
\begin{align*}
&\tilde d_{{\rm opt}}:=\max_{P_0,\ldots,P_{N-1}\in {\mathbb S}_+^{n+m}}\,\,J_d (\{
P_k,F_k\}_{k=0}^{N-1})\\
&{\rm s.t.}\\
&\Gamma (F_k,P_k)\succeq P_{k-1},k \in\{1,\ldots, N-1\},\\
&\Gamma (0,P_N)\succeq P_{N-1},\\
&\begin{bmatrix} 0\\I\\
\end{bmatrix}^T P_k \begin{bmatrix}
 0\\ I\\
\end{bmatrix} \succ 0,\quad F_k=-P_{k,22}^{-1} P_{k,12}^T,\quad k\in \{0,\ldots N-1\},
\end{align*}
where
\begin{align*}
J_d(\{P_k,F_k\}_{k=0}^{N-1}):=& Tr\left( \begin{bmatrix}
 I\\ F_0\\
\end{bmatrix}W_f \begin{bmatrix}
   I\\ F_0\\
\end{bmatrix}^T P_0\right)+ \sum_{k=1}^{N} Tr\left( \begin{bmatrix}
   I\\F_k\\
\end{bmatrix} W \begin{bmatrix}
   I\\F_k\\
\end{bmatrix}^T P_k\right)\\
\Gamma (F,P):=&\begin{bmatrix}
A^T\\B^T\\
\end{bmatrix} \begin{bmatrix}
I\\F\\
\end{bmatrix}^T P \begin{bmatrix}
I\\F\\
\end{bmatrix} \begin{bmatrix}
A^T\\B^T\\
\end{bmatrix}^T + \begin{bmatrix}
   Q & 0\\0 & R\\
\end{bmatrix},
\end{align*}
and
\begin{align}
&P_k=\begin{bmatrix}
   P_{k,11}& P_{k,12}\\
   P_{k,12}^T& P_{k,22}\\
\end{bmatrix},\quad P_N=\begin{bmatrix}
   Q_f& 0\\ 0 & 0\\
\end{bmatrix}\nonumber
\end{align}
\end{problem}

\cref{dual-LQR-inequality-formulation} is equivalent to~\cref{dual-problem1} in the sense that the optimal objective function values are identical, and an optimal solution of~\cref{dual-LQR-inequality-formulation} is identical to the corresponding optimal solution of~\cref{dual-problem1}. \cref{dual-LQR-inequality-formulation} is a convex optimization problem (SDP problem), whose solution can be easily found by existing convex optimization tools. The results are formally summarized in the following theorem.
\begin{theorem}\label{theorem2}
$d_{\rm opt}=\tilde d_{\rm opt}$ and an optimal
point of~\cref{dual-LQR-inequality-formulation} is $\{ P_k\}_{k=0}^{N-1}$ with
\begin{align}
&P_k=\begin{bmatrix}
   Q + A^T X_{k+1} A & A^T X_{k+1} B\\
   B^T X_{k+1} A & R+B^T X_{k+1} B\\
\end{bmatrix}\label{dual-optimal-point}
\end{align}
for $k \in \{0,\ldots,N-1\}$, where $X_{k+1},k \in \{0,\ldots,N-1\}$ are the solution given in~\cref{theorem1b}.
\end{theorem}
\begin{proof}
We first define the set
\begin{align*}
&{\cal F}: = \left\{ {P \in {\mathbb S}_+^{n+m}
:\,\,\begin{bmatrix}
   0  \\
   {I_m }  \\
\end{bmatrix}^T P \begin{bmatrix}
   0  \\
   {I_m }  \\
\end{bmatrix} \succ 0 }\right\}.
\end{align*}

Form the solution of the KKT
condition in~\cref{theorem1b}, there exists a unique
dual optimal point~\eqref{dual-optimal-point}, which satisfies
$\begin{bmatrix}
   0  \\
   {I_m }  \\
\end{bmatrix}^T P_k \begin{bmatrix}
   0  \\
   {I_m }  \\
\end{bmatrix} \succ 0$. This ensures that the optimal objective function value of the dual
problem in~\cref{dual-problem1} is not changed when the constraints
$\begin{bmatrix}
 0  \\
I_m \\
\end{bmatrix}^T P_k \begin{bmatrix}
   0  \\
 I_m \\
\end{bmatrix} \succ 0,k \in \{ 0,1, \ldots ,N - 1\} $ is added. we can consider~\cref{dual-problem1} with its solution restricted to ${\cal F}$ as follows:
\begin{align}
&\sup_{\scriptstyle P_k  \in {\cal
F},{\bar P}_k \succeq 0 \hfill \atop \scriptstyle k \in \{
0, \ldots ,N - 1\}  \hfill} \mathop {\inf }\limits_{\{ S_k
,F_k \} _{k = 0}^{N - 1}} L(\{ (S_k ,F_k ,P_k ,{\bar P}_k )\} _{k=0}^{N - 1})\label{eq7}
\end{align}

Now, let us focus on the term in the Lagrangian~\eqref{eq1}, i.e.,
$\sum_{k = 1}^{N - 1} {Tr(\Gamma (F_k ,P_k
)S_{k-1} )}$, which can be represented by
\begin{align*}
\sum_{k = 1}^{N - 1} {Tr(\Gamma (F_k ,P_k
) S_{k - 1} )} = \sum_{k = 1}^{N - 1} {{\mathbb E}\left({\begin{bmatrix}
   z  \\
   F_k z\\
\end{bmatrix}^T P_k \begin{bmatrix}
   z\\
   F_k z\\
\end{bmatrix} } \right) + \sum_{k = 1}^{N - 1} {Tr\left( {\begin{bmatrix}
   Q & 0 \\
   0 & R \\
\end{bmatrix} S_{k - 1}}\right)}}
\end{align*}
where $z: = Ax(k - 1) + Bu(k - 1)$. If $\begin{bmatrix}
   0  \\
   {I_m }  \\
\end{bmatrix}^T P_k \begin{bmatrix}
   0  \\
   {I_m }  \\
\end{bmatrix} \succ 0$, then it is minimized with respect to $F_k$ when $F_k =  - P_{22,\,k}^{ - 1}
P_{12,\,k}^T$.

Therefore, \eqref{eq7} is equivalent to
\begin{align*}
&\sup_{\scriptstyle P_k \in {\cal
F},\,{\bar P}_k \succeq 0 \hfill \atop \scriptstyle k \in \{
0, \ldots ,\,N - 1\}  \hfill} \inf_{\{ S_k
\} _{k = 0}^{N - 1} } L(\{ (S_k ,\bar F_k ,P_k ,{\bar P}_k )\} _{k = 0}^{N - 1} )
\end{align*}
where $\bar F_k : =  - P_{22,k}^{ - 1} P_{12,k}^T$. Since
$\inf_{\{ S_k\}_{k=0}^{N-1}}L (\{ (S_k ,\bar F_k ,P_k ,\,{\bar P}_k )\}
_{k = 0}^{N - 1} )$ has a finite value only when $\Gamma (F_k^*
,P_k ) - {\bar P}_{k - 1} = P_{k - 1} ,k
\in \{ 1,2, \ldots N - 1\}$ and $\Gamma (0,P_N ) -
{\bar P}_{N - 1} = P_{N - 1}$, the problem~\eqref{eq7}
can be formulated as
\begin{align*}
&{\rm max}_{\scriptstyle P_0, \ldots
,P_{N - 1} \in {\mathbb S}_+^n \hfill \atop \scriptstyle
{\bar P}_0, \ldots ,{\bar P}_{N-1} \in {\mathbb S}_+^n
\hfill}\,J_d (\{ P_k,\bar F_k \} _{k = 0}^{N - 1})\\
&{\rm s.t.}\\
&\Gamma (\bar F_k,P_k ) - {\bar P}_{k - 1}  =
P_{k - 1} ,\quad k \in \{ 1,2, \ldots N - 1\}\\
&\Gamma (0,P_N ) - {\bar P}_{N - 1}=P_{N -1}
\end{align*}
or equivalently,
\begin{align*}
&{\rm max}_{P_0 , \ldots ,P_{N-1}
\in {\mathbb S}_+^n } \,\,J_d (\{ P_k,\bar F_k \}_{k=
0}^{N - 1})\\
&{\rm s.t.}\\
&\Gamma (\bar F_k,P_k ) \succeq P_{k-1} ,\quad k
\in \{ 1,2, \ldots N - 1\}\\
&\Gamma (0,P_N ) \succeq P_{N - 1}
\end{align*}

This completes the proof.

\end{proof}

Note that the approaches given in this paper can be easily extended to linear time-varying systems. In the next section, we study the decentralized LQG problem by combining the developments in this section and the results in~\cite{de2002extended}.

\section{Decentralized LQG performance analysis and design}

The structure of the optimization in~\cref{theorem1} allows us to derive a sufficient but simple convex relaxation for designing a decentralized LQG controller. Consider the stochastic LTI system composed of $M$ interconnected subsystems
\begin{align}
&x_i(k+1) = \sum_{j=1}^M {A_{ij} x_j (k)}+B_i u_i(k)+w_i(k)\label{interconnected-system}
\end{align}
for $i\in \{1,\ldots,M\}$, where $k \in {\mathbb N}$ is the time,
$x_i (k) \in {\mathbb R}^{n_i}$ is the state vector, $u_i(k)\in
{\mathbb R}^{m_i}$ is the control vector, $x_i (0)\sim {\cal
N}(0,W_f)$ and $w_i(k) \sim {\cal N}(0,W)$ are mutually
independent Gaussian random vectors. Let us define
\begin{align}
&x(k)=\begin{bmatrix}
   x_1(k)\\ \vdots\\
   x_M (k)\\
\end{bmatrix}, u(k) = \begin{bmatrix}
   {u_1 (k)}  \\
    \vdots   \\
   {u_M (k)}  \\
\end{bmatrix}, w(k) = \begin{bmatrix}
   {w_1 (k)}  \\
    \vdots   \\
   {w_M (k)}  \\
\end{bmatrix}.\label{eq12}
\end{align}

Then, the system dynamics~\eqref{interconnected-system} can be written as
\begin{align*}
&x(k + 1) = Ax(k) + Bu(k) + w(k)
\end{align*}
where
\begin{align*}
A = \begin{bmatrix}
   {A_{11} } &  \cdots  & {A_{1M} }  \\
    \vdots  &  \ddots  &  \vdots   \\
   {A_{M1} } &  \cdots  & {A_{MM} }  \\
\end{bmatrix} \in {\mathbb R}^{n \times n}, \quad B = {\rm diag}(B_1 , \ldots ,\,B_M ) \in {\mathbb R}^{n \times m},
\end{align*}
$n = n_1  +  \cdots  + n_M $, and $m = m_1  +  \cdots  + m_M$. Now, we formally state the decentralized state-feedback LQG problem in the sequel.
\begin{problem}[Decentralized LQG problem]\label{decentralized-LQR-problem2}
Solve
\begin{align*}
J_{\cal K}^* :=& \min_{F_k
\in {\mathbb R}^{m \times n},k \in \{
0,\,1, \ldots ,\,N - 1\}} \,{\mathbb E}(x(k)^T Q_f x(k))+ \sum_{k=0}^{N-1} {{\mathbb E}\left( {\begin{bmatrix}
   {x(k)}  \\
   {u(k)}  \\
\end{bmatrix}^T \begin{bmatrix}
   Q & 0  \\
   0 & R  \\
\end{bmatrix} \begin{bmatrix}
   x(k) \\
   u(k) \\
\end{bmatrix}} \right)}\\
{\rm subject \,\, to}&\\
x(k + 1) =& Ax(k) + Bu(k) + w(k)\\
u(k)=& F_k x(k),\quad F_k  \in {\cal K}
\end{align*}
where ${\cal K}$ is a linear subspace defined as ${\cal K}: = \{ K
\in {\mathbb R}^{m \times n} :\,K = {\rm diag}(F_1 ,\,F_2 , \ldots
,\,F_M ),\,F_i \in {\mathbb R}^{m_i \times n_i } ,\,i \in \{ 1,
\ldots ,\,M\} \}$.
\end{problem}

Equivalently, the problem can be converted into~\cref{primal-LQR} and~\cref{primal-LQR-inequality-formulation} with the
additional constraint $F_k  \in {\cal K}, k \in \{ 0,1,\ldots
,N - 1\}$. The problem is a non-convex structured state-feedback design problem. When $F_k  \in {\cal K}, k \in \{
0,1, \ldots ,N - 1\}$ is given, then its exact cost can be
evaluated using a convex optimization as follows.
\begin{proposition}
Let $F_k  \in {\cal K},k \in \{ 0,1, \ldots ,N - 1\}$ be
given. The cost corresponding to the given structured static
state-feedback gain is $J^* (F_0,\ldots ,F_{N - 1} ): = J_p (\{
S_k \} _{k = 0}^{N - 1})$ where $S_k  = \Phi (F_k
,S_{k - 1} ),k \in \{ 1, \ldots ,N-1\} $ with
$S_0  = \begin{bmatrix}
   I_n\\
   F_0\\
\end{bmatrix} W_f \begin{bmatrix}
   I_n \\
   F_0 \\
\end{bmatrix}^T$.
\end{proposition}

The cost can be also evaluated using~\cref{primal-LQR-inequality-formulation}, which is simply an
SDP if $F_k\in {\cal K},\,k \in \{ 0,\ldots,N-1\}$ are
constants. Next, motivated by the LMI-based decentralized control design
method in~\cite{de2002extended}, we suggest a simple convex
relaxation of~\cref{decentralized-LQR-problem2}.
\begin{problem}\label{decentralized-LQR-problem}
Solve
\begin{align*}
&(S_k^*,L_k^*,G_k^*)_{k=0}^{N - 1}:= \argmin_{ S_k \in {\mathbb S}^{n + m}
,L_k\in {\mathbb R}^{n \times m},G_k \in {\mathbb R}^{n
\times n}}\,f^p (\{ S_k \}_{k=0}^{N-1})\\
&{\rm subject\,\, to}\\
&\left[ {\begin{array}{*{20}c}
   {S_k } & *  \\
   {\left[ {\begin{array}{*{20}c}
   {G_k } & {L_k }  \\
\end{array}} \right]} & {G_k  + G_k^T  - \left[ {\begin{array}{*{20}c}
   {A^T }  \\
   {B^T }  \\
\end{array}} \right]^T S_{k - 1} \left[ {\begin{array}{*{20}c}
   {A^T }  \\
   {B^T }  \\
\end{array}} \right] - W}  \\
\end{array}} \right] \succeq  0,\quad \forall k \in \{ 1,2,\ldots,N-1\}\\
&\left[ {\begin{array}{*{20}c}
   {S_0 } & *  \\
   {\left[ {\begin{array}{*{20}c}
   {G_0 } & {L_0 }  \\
\end{array}} \right]} & {G_0  + G_0^T  - W_f }  \\
\end{array}} \right]\succeq 0\\
&G_k= {\rm diag}(G_{k,1}, \ldots,G_{k,M} ),\quad L_k= {\rm diag}(L_{k,1}, \ldots,L_{k,M})\\
& L_{k,i}\in {\mathbb R}^{n_i \times m_i },\quad G_{k,\,i} \in
{\mathbb R}^{n_i \times n_i}
\end{align*}
\end{problem}

\cref{decentralized-LQR-problem} is a convex optimization problem (SDP problem), whose solution can be easily found using existing tools. Once its solution is found, then a suboptimal state feedback gain can be recovered from the solution.
\begin{proposition}\label{proposition:1}
Let $(S_k^* ,L_k^*,G_k^* )_{k = 0}^{N - 1}$ be an
optimal point of~\cref{decentralized-LQR-problem}, and let $\tilde J_{\cal
K}^*$ be the corresponding optimal objective function value. Then, $J_{\cal K}^*
\le \tilde J_{\cal K}^* $ is satisfied under the decentralized
control policy $u_i(k) = (L_{k,i}^* )^T (G_{k,i}^*)^{-T}
x_i(k)$ for all $k \in \{ 0,1, \ldots ,N - 1\} $ and $i \in \{
1,2, \ldots ,M\} $.
\end{proposition}
\begin{proof}

Since $W\succ 0$, it is easy to see that if the SDP is feasible, then $G_k  +G_k^T\succ 0$, implying that $G_k$ is invertible.
Pre- and post-multiplying both sides of the inequalities in~\cref{decentralized-LQR-problem} by
\begin{align*}
&\begin{bmatrix}
   - I & 0  \\
   0 & -I  \\
   I & G_k^{-1} L_k \\
\end{bmatrix}^T ,\,k \in \{ 0,1, \ldots ,N - 1\}
\end{align*}
and its transpose yield
\begin{align*}
\Phi (F_k,S_{k-1}) \preceq & S_k\quad k \in \{1,2,\ldots,N-1\}\\
\begin{bmatrix}
   I \\
   F_0\\
\end{bmatrix}W_f \begin{bmatrix}
   I\\
   F_0\\
\end{bmatrix}^T \preceq & S_0
\end{align*}
with $F_k = L_k^T G_k^{-T},k \in \{0,1, \ldots ,N - 1\}$. By using~\cref{theorem1}, one concludes that $J_{\cal K}^* \le
\tilde J_{\cal K}^*$ is satisfied under the policy $u(k) = F_k^*
x(k),k \in \{ 0,\ldots,N-1\}$. Since $F_k^*$ has a
block diagonal structure according to the state and input
partitions in~\eqref{eq12}, the desired result can be obtained.

\end{proof}

It can be readily proved that $J^*_p \le J_{\cal K}^* \le J^*
(F_0^*,\ldots,F_{N - 1}^*) \le \tilde J_{\cal K}^*$ holds, where
\begin{enumerate}
\item $J^*_p$ is the optimal cost corresponding to the centralized full state-feedback in~\cref{primal-LQR}
\item $J_{\cal K}^*$ is the true optimal cost obtained by solving~\cref{decentralized-LQR-problem2}
\item $J^*(F_0^*,\ldots ,F_{N-1}^* )$ is the exact cost evaluated
using $F_0^*,\ldots ,F_{N-1}^*$ obtained from~\cref{decentralized-LQR-problem}
\item $\tilde J_{\cal K}^*$ is the optimal objective value of~\cref{decentralized-LQR-problem}
\end{enumerate}

Note that $\tilde J_{\cal K}^* \geq J^*(F_0^*,\ldots ,F_{N-1}^* )$ due to the inherent conservatism of the SDP in~\cref{decentralized-LQR-problem}. A simple example is given in the sequel.
\begin{example}\label{ex1}
Consider the interconnected system
\begin{align*}
&x_1 (k + 1) = A_{11} x_1(k) + A_{12} x_2 (k) + B_1 u_1 (k) +
w_1(k)\\
&x_2 (k + 1) = A_{21} x_1(k) + A_{22} x_2 (k) + B_2 u_2 (k) +
w_2(k)
\end{align*}
where
\begin{align*}
A_{11}  =& \begin{bmatrix}
   0.8220 & -0.0898  \\
   -0.2389 & 0.9358  \\
\end{bmatrix},\quad A_{12}  = \begin{bmatrix}
   0.4860 & -0.1820 \\
   0.1680 & -0.3143 \\
\end{bmatrix}\\
A_{21}  =& \begin{bmatrix}
   0.1891 & -0.3195 \\
   0.2067 & -0.6610 \\
\end{bmatrix},\quad A_{22}  = \begin{bmatrix}
   - 0.6404 & 1.4540 \\
   0.2067 & - 0.6610 \\
\end{bmatrix}\\
B_1  =& \begin{bmatrix}
   { - 0.3505}  \\
   { - 1.9788}  \\
\end{bmatrix}\quad B_2  = \begin{bmatrix}
   { - 0.4901}  \\
   { - 0.0515}  \\
\end{bmatrix}
\end{align*}

Solving~\cref{decentralized-LQR-problem} with $Q = Q_f  =
I_n ,R = I_n ,W = 0.01I_n ,W_f  = I_n$, and $N=30$ yields
$\tilde J_{\cal K}^*  = 19.6799$ and $J^* (F_0^*,\ldots ,F_{N
- 1}^* ) = 18.0598$. On the other hand, the optimal cost
corresponding to the centralized LQG (full state-feedback) is $J_p^*  = 16.2610$.
Therefore, one concludes $J_p^*= 16.2610 \le J_{\cal K}^* \le
J^* (F_0^*,\ldots ,F_{N - 1}^* ) = 18.0598$. The time
histories of the state under the obtained decentralized control
policy is shown in~\cref{fig1} and the histogram of the cost
of 3000 simulations is plotted in~\cref{fig2}.
\begin{figure}[h!]
\centering\epsfig{figure=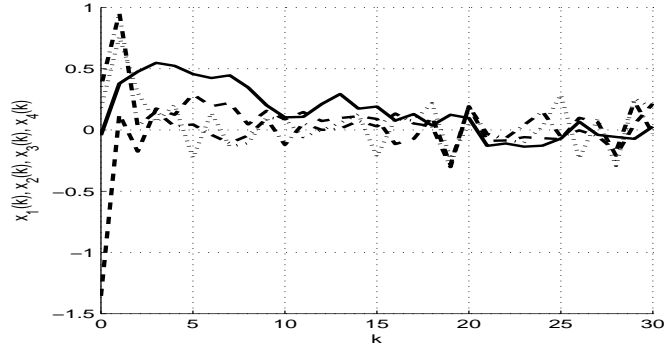,
height=5cm,width=12cm}\caption{Time histories of the
state under the obtained decentralized control policy.
}\label{fig1}
\end{figure}
\begin{figure}[h!]
\centering\epsfig{figure=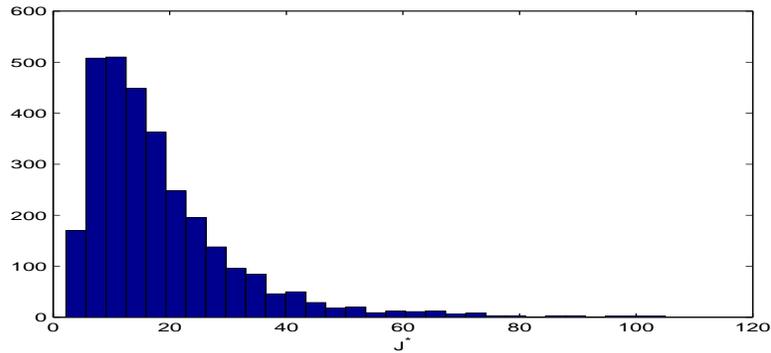,
height=5cm,width=12cm}\caption{Cost histogram of 3000
simulations}\label{fig2}
\end{figure}
\end{example}

\section{Conclusion}

In this paper, we have presented a new SDP formulation of the finite-horizon LQG problem and its dual. The proposed primal problem efficiently decouples the system matrices and the gain matrix. This fact allows us to develop new convex relaxations of non-convex structured control design problems such as the decentralized control problem. Besides, we are expected to gain new insights on the LQG problem through this study. Numerical examples have demonstrated the effectiveness of the proposed SDP formulations.

\section{Acknowledgement}
D. Lee is thankful to J. Hu and D. Kim for their fruitful comments on this paper.

\bibliographystyle{IEEEtran}
\bibliography{reference}

\

\end{document}